\documentclass[pdflatex]{sn-jnl}
\usepackage{graphicx}%
\usepackage{multirow}%
\usepackage{amsmath,amssymb,amsfonts}%
\usepackage{amsthm}%
\usepackage{mathrsfs}%
\usepackage[title]{appendix}%
\usepackage{xcolor}%
\usepackage{textcomp}%
\usepackage{manyfoot}%
\usepackage{booktabs}%
\usepackage{algorithm}%
\usepackage{algorithmicx}%
\usepackage{algpseudocode}%
\usepackage{listings}%

\newtheorem{theorem}{Theorem}
\newtheorem{prop}[theorem]{Proposition}
\newtheorem{lem}[theorem]{Lemma}

\begin{document}

\title[Discrepancy estimate]{Discrepancy estimates related to the fractional parts of $b^n/n$.}
\author{\fnm{Martin} \sur{Lind}}
\email{martin.lind@kau.se}
\affil{\orgdiv{Dept. of Mathematics and Computer Science}, \orgname{Karlstad University},\orgaddress{ \street{Universitetsgatan 2}, \city{Karlstad}, \postcode{65188}, \country{Sweden}}}

\abstract{We prove a discrepancy estimate related to the sequence of fractional parts of $b^n/n$. This improves an earlier result of Cilleruelo et al.}



\maketitle

\section{Introduction}\label{sec1}

Let $b\in\mathbb{N}, b\ge2$. In 2013, Cilleruelo et al. \cite{CKLRS} proved that
\begin{equation}
    \label{Seq}
    \left\{\frac{b^n\pmod{n}}{n}:n\in\mathbb{N}\right\}
\end{equation}
is dense in $[0,1]$. (See also \cite{D2} for a number of interesting related results.) 

For $A\subset\mathbb{N}$, we set
\begin{equation}
    \nonumber
    \mathcal{S}_b(A)=\left\{\frac{b^n\pmod{n}}{n}:n\in A\right\}.
\end{equation}
Note in particular that the set (\ref{Seq}) is simply $\mathcal{S}_b(\mathbb{N})$. Let $\mathbb{P}=\{2,3,\ldots\}$ denote the prime numbers and set 
\begin{equation}
    \nonumber
    \mathcal{A}=\left\{pq:p,q\in\mathbb{P}, p>b^q\right\}.
\end{equation}
The main result of \cite{CKLRS} is an estimate of the discrepancy of $\mathcal{S}_b(\mathcal{A})$. Denote $\mathcal{A}_N=\mathcal{A}\cap[1,N]$, then
\begin{equation}
    \label{discrEstimate1}
    D(\mathcal{S}_b(\mathcal{A}_N))=\mathcal{O}\left(\frac{\log\log\log\log(N)}{\log\log\log(N)}\right).
\end{equation}
where $D(\mathcal{S}_b(\mathcal{A}_N))$ denotes the \emph{discrepancy} of $\mathcal{S}_b(\mathcal{A}_N)$
(see Section \ref{sec2} below).
In particular, it follows from (\ref{discrEstimate1}) that $\mathcal{S}_b(\mathcal{A})$ is uniformly distributed modulo 1 and this implies the density of $\mathcal{S}_b(\mathcal{A})$ in $[0,1]$.

Unaware of the work \cite{CKLRS}, the author studied properties of $\mathcal{S}_b(\mathcal{A})$ from a different point of view in \cite{M2}. When informed of the paper \cite{CKLRS}, we found that some observations from \cite{M2} could be used to improve on (\ref{discrEstimate1}). The main result of this note is the following.
\begin{theorem}
\label{discrepancyThm}
There holds
\begin{equation}
    \label{discrEstimate2}
    D(\mathcal{S}_b(\mathcal{A}_N))=\mathcal{O}\left(\frac{1}{\log\log\log(N)}\right).
\end{equation}
\end{theorem}
The improvement (\ref{discrEstimate2}) is not due to any sharper number-theoretic inequalities. In fact, we use the same estimates as in \cite{CKLRS}. Rather, we employ a different strategy to estimate the discrepancy. Instead of using the Erd\H{o}s-Tur\'{a}n inequality and exponential sums as in \cite{CKLRS}, we use a sort of triangle inequality (Lemma \ref{triangleIneq}) to decompose $\mathcal{S}_b(\mathcal{A}_N)$ into well-structured subsequences. By combining a number of basic facts about discrepancy with some observations from \cite{M2} (Proposition \ref{countProp} in particular) and a variant of the Siegel-Walfisz theorem, we obtain in Lemma \ref{discrepancyLemma} an estimate of the discrepancy of each subsequence and these estimates allow us to establish Theorem \ref{discrepancyThm}.
In connection with this, we mention our previous work \cite{M1} where a similar strategy based on Lemma \ref{triangleIneq} was used to find optimal discrepancy decay rates.

\section{Auxiliary results}\label{sec2}
\subsection{Discrepancy}

For a finite set $A$, we denote by $|A|$ the cardinality of $A$.
Let $S=\{x_1,x_2,\ldots,x_M\}\subset[0,1]$ be a finite sequence. The \emph{extreme discrepancy} of $S$ is defined by 
\begin{equation}
    \nonumber
    D(S)=\sup_{J\subseteq[0,1]}\left|\frac{A_S(J)}{M}-\lambda(J)\right|,
\end{equation}
where $A_S(J)=|\{n:x_n\in J\}|$ and $\lambda$ is the linear Lebesgue measure. Similarly, the \emph{star discrepancy} of $S$ is defined by 
\begin{equation}
    \nonumber
    D^*(S)=\sup_{r>0}\left|\frac{A_S([0,r])}{M}-r\right|.
\end{equation}
It is well-known (see e.g. \cite{DP}, Chapter 3) that
\begin{equation}
    \nonumber
    D^*(S)\le D(S)\le 2D^*(S),
\end{equation}
hence it is sufficient to only consider $D^*(S)$.

\begin{lem}
\label{discreteLemma}
    Let $R\in\mathbb{N}$ and $S=\{x_1,x_2,\ldots,x_M\}$ be a finite sequence such that the elements of $S$ only attain values in the set $\{k/R: k=0,1,\ldots, R-1\}$. Assume that 
    \begin{equation}
        \nonumber
        |\{n:x_n=k/R\}|=\alpha_kM+\epsilon_k
    \end{equation}
    where $\alpha_k\ge0~~(k=0,1,\ldots,R-1)$,  $\alpha_1=\alpha_2=\ldots=\alpha_{R-1}$ and
    \begin{equation}
        \nonumber
        \sum_{k=0}^{R-1}\alpha_k=1.
    \end{equation}
    Then there exists an absolute constant $C>0$ such that
    \begin{equation}
        \nonumber
        MD^*(S)\le\max\left\{\alpha_0M,\frac{M}{R}\right\}+\sum_{k=0}^{R-1}|\epsilon_k|.
    \end{equation}
\end{lem}
\begin{proof}
   Take any $J_r=[0,r]$ and let $j=\lfloor rR\rfloor$, so that $M\mu(J_r)=jM/R+M\delta$ for some $\delta\in[0,1/R]$. We have
   \begin{equation}
       \nonumber
        A_S(J_r)=M\sum_{k=0}^j\alpha_k+\sum_{k=0}^j\epsilon_k
   \end{equation}
    so
   \begin{equation}
       \nonumber
        \left|A_S(J_r)-M\mu(J_r)\right|\le\left|M\alpha_0+M\sum_{k=1}^j\left(\alpha_k-\frac{1}{R}\right)-M\delta\right|+\sum_{k=1}^j|\epsilon_k|
   \end{equation}
    The first term of the expression at the right-hand side above is either increasing or decreasing in $\delta$, hence we have
   \begin{equation}
       \nonumber
        MD^*(S)\le\max_{j=0,\ldots, R-1}\left|M\alpha_0+M\sum_{k=1}^j\left(\alpha_k-\frac{1}{R}\right)\right|+\sum_{j=0}^{R-1}|\epsilon_k|
   \end{equation}
    where $j=0$ means that the sum is 0. The maximum of the first term is attained either at $k=0$ or $k=R-1$, since the terms of the sum have the same sign (due to the fact that $\alpha_1=\ldots=\alpha_{R-1}$. Further, $\alpha_0+(R-1)\alpha_1=1$ so
    \begin{eqnarray}
        \nonumber
        \max_{j=0,\ldots, R-1}\left|M\alpha_0+M\sum_{k=1}^j\left(\alpha_k-\frac{1}{R}\right)\right|&=&\max\left\{M\alpha_0,\left|M\alpha_0+M\sum_{k=1}^{R-1}\left(\alpha_k-\frac{1}{R}\right)\right|\right\}\\
        \nonumber
        &=&\max\left\{M\alpha_0,\left|M-\frac{M(R-1)}{R}\right|\right\}\\
        \nonumber
        &=&\max\left\{M\alpha_0,\frac{M}{R}\right\}
    \end{eqnarray}
    Consequently,
    $$
    MD^*(S)\le\max\left\{M\alpha_0,\frac{M}{R}\right\}+\sum_{k=0}^{R-1}|\epsilon_k|.
    $$
\end{proof}

\begin{lem}[\cite{DP}, Chapter 3]
    \label{triangleIneq}
    Assume that $S=\cup_j S_j$ where $S_i\cap S_j=\emptyset$. Denote $M_j=|S_j|$ and $M=|S|=\sum_jM_j$. Then
    \begin{equation}
        \nonumber
        MD^*(S)\le\sum_{j=1}^KM_jD^*(S_j).
    \end{equation}
\end{lem}
\begin{lem} [\cite{NW}, Chapter 4]
    \label{continuity}
    Let $S'=\{x_1,x_2,\ldots,x_M\}$ and $S''=\{y_1,y_2,\ldots,y_M\}$ such that
    \begin{equation}
        \nonumber
        |x_j-y_j|\le\epsilon
    \end{equation}   
    for $j=1,2,\ldots,M$. Then
    \begin{equation}
        \nonumber
        |D^*(S')-D^*(S'')|<\epsilon.
    \end{equation}
\end{lem}

\subsection{Primes in arithmetic progressions}
Denote
\begin{equation}
    \nonumber
    Z_k=\{r\in\mathbb{Z}^*_{q(q-1)}: b^r\equiv kr+b\pmod{q}\}.
\end{equation}
For any $p\in\mathbb{P}$ with $p>b^q$ and $p\equiv r\pmod{q(q-1)}$ for some $r\in Z_k$, there holds
\begin{equation}
    \label{pqEstimate}
    \left|\frac{b^{pq}\pmod{pq}}{pq}-\frac{k}{q}\right|<\frac{1}{q},
\end{equation}
see \cite{M2}. 

Let ${\rm ord}_q(b)=|\langle b\rangle|$ where $\langle b\rangle$ is the subgroup of $\mathbb{Z}^*_q$ generated by $b$.
In \cite{M2}, we proved the following proposition.
\begin{prop}
\label{countProp}
For $q\in\mathbb{P}$ there holds
    \begin{equation}
        \nonumber
        |Z_k|=\varphi(q-1)-m_b(q)\quad(k=1,2,\ldots,q-1),
    \end{equation}
    and
    \begin{equation}
        \nonumber
        |Z_0|=(q-1)m_b(q)
    \end{equation}
    where
    \begin{equation}
        \nonumber
        m_b(q)=|\{r\in\mathbb{Z}^*_{q-1}: r\equiv 1\pmod{{\rm ord}_q(b)}\}|.
    \end{equation}
\end{prop}
We shall need to estimate the number of primes in certain arithmetic progressions. Denote
$$
\pi(N;q(q-1),r)=|\{p\in\mathbb{P}: p\le N, p\equiv r\pmod{q(q-1)}\}.
$$
We use the following consequence of the Siegel-Walfisz theorem (see \cite{CKLRS} and the reference given there):
\begin{equation}
    \label{siegel}
    \pi(N;q(q-1),r)=\frac{\pi(N)}{\varphi(q(q-1))}+\mathcal{O}\left(\frac{N}{(\log(N))^A}\right)
\end{equation}
for any $A>0$ and $N\ge2$. (The implied constant in (\ref{siegel}) depends on $A$.) Here, $\pi(N)$ is the prime counting function. In particular, there is asymptotically the same amount of primes in the progression $r+nq(q-1)$ for each $r\in\mathbb{Z}^*_{q(q-1)}$. More precisely, (\ref{siegel}) and Proposition \ref{countProp} imply that
\begin{align}
    \nonumber
    |\{p\in\mathbb{P}:p\le N, \exists r\in Z_k\text{ such that } p\equiv r\pmod{q(q-1)}\}|= \\
    \nonumber
    =\frac{|Z_k|}{\varphi(q(q-1))}\pi(N)+\mathcal{O}\left(\frac{N|Z_k|}{(\log(N))^A}\right)
\end{align}

\section{Proof of Theorem \ref{discrepancyThm}}\label{sec3}

Denote 
\begin{equation}
    \nonumber
    F_{q,N}=\{p\in\mathbb{P}:b^q<p\le N/q\}\quad\text{and}\quad \mathcal{F}_{q,N}=\{pq:p\in F_{q,N}\},
\end{equation}
then
    \begin{equation}
        \nonumber
        \mathcal{A}_N=\bigcup_{q\in\mathbb{P}}\mathcal{F}_{q,N}.
    \end{equation}
(Note that $F_{q,N}=\emptyset$ for $q$ sufficiently large.)
Define
\begin{equation}
    \nonumber
    F_k=\{p\in F_{q,N}: \exists r\in Z_k\text{ such that } p\equiv r\pmod{q(q-1)}\}.
\end{equation}
for $k=0,1,\ldots,q-1$.

\begin{lem}
    \label{numberLemma}
    For each $k\in\{0,1,\ldots,q-1\}$, there holds
    \begin{equation}
        \nonumber
        |F_k|=\frac{|Z_k|}{\varphi(q(q-1))}|F_{q,N}|+\epsilon_k
    \end{equation}
    where 
    \begin{equation}
        \nonumber
        \sum_{k=0}^{q-1}|\epsilon_k|\le \frac{CN}{q^2\log(N)}.
    \end{equation}
\end{lem}
\begin{proof}
Observe that
$$
|F_k|=|Z_k|\left(\pi(N;q(q-1),r)-\pi(b^q,q(q-1),r)\right).
$$
Taking $A=4$ in (\ref{siegel}), we obtain
\begin{eqnarray}
    \nonumber
    |F_k|&=&\frac{|Z_k|}{\varphi(q(q-1))}\left(\pi(N/q)-\pi(b^q)\right)+\epsilon_k\\
    \nonumber
    &=&\frac{|Z_k|}{\varphi(q(q-1))}|F_{q,N}|+\epsilon_k,
\end{eqnarray}
where 
\begin{eqnarray}
    \nonumber
    |\epsilon_k|&\le& C|Z_k|\left(\frac{N}{q(\log(N/q))^4}+\frac{b^q}{q^4}\right)\le 2C|Z_k|\frac{N}{q(\log(N/q))^4}
\end{eqnarray}
since $x\mapsto x/(\log(x))^4$ is increasing. Note that
\begin{equation}
    \nonumber
    \sum_{k=0}^{q-1}|\epsilon_k|\le\frac{2CN}{q(\log(N/q))^4}\sum_{k=0}^{q-1}|Z_k|=\frac{2CN}{q(\log(N/q))^4}\varphi(q(q-1))\le\frac{2CNq^2}{q(\log(N/q))^4}   
\end{equation}
Since $b^q<N/q$, we have $q\le\log(N/q)$. Furthermore, $q^2< qb^q< N$, so $N/q>\sqrt{N}$. Consequently,
$$
    \frac{1}{\log(N/q)}\le\frac{1}{q}\quad\text{and}\quad \frac{\log(N)}{2}\le\log(N/q)\le\log(N)
$$
and therefore
\begin{equation}
    \nonumber
    \sum_{k=0}^{q-1}|\epsilon_k|\le\frac{2CNq^2}{q(\log(N/q))^4}\le\frac{4CN}{q^2\log(N)}.
\end{equation}
\end{proof}
Denote by
\begin{equation}
    \nonumber
    n(q,N)=|F_{q,N}|
\end{equation}
\begin{lem} 
There exists an absolute constant $C$ such that for any $q\in\mathbb{P}$, $N>b^q$, there holds
\label{discrepancyLemma}
    \begin{equation}
        \label{mainDiscrEstimate}
        n(q,N)D^*(\mathcal{S}_b(\mathcal{F}_{q,N}))\le C\left(\frac{\log\log(q)}{\log(q)}n(q,N)+\frac{N}{q^2\log(N)}\right).
    \end{equation}
\end{lem}
\begin{proof}
By (\ref{pqEstimate}), for any $p>b^q$ there is a $k\in\{0,1,\ldots,q-1\}$ such that
\begin{equation}
    \label{diff}
    \left|\frac{b^{pq}\pmod{pq}}{pq}-\frac{k}{q}\right|<\frac{1}{q}
\end{equation}
holds. Furthermore, for a specific $k$ the estimate  (\ref{diff}) holds if and only if $p\equiv r\pmod{q(q-1)}$ where $r\in Z_k$.

For $p\in F_{q,N}$ we define $a_p=k/q$ if $p\in F_k$. Set $S'=\{a_p:p\in F_{q,N}\}$. Then $|S'|=|F_{q,N}|$ and
\begin{equation}
    \nonumber
    \left|\frac{b^{pq}\pmod{pq}}{pq}-a_p\right|<\frac{1}{q}
\end{equation}
for each $p\in F_{q,N}$. By Lemma \ref{continuity}, there holds
\begin{equation}
    \label{mainLemmaEq1}
    D^*(S')-\frac{1}{q}<D^*(\mathcal{S}_b(\mathcal{F}_{q,N})<D^*(S')+\frac{1}{q}
\end{equation}
We shall now use Lemma \ref{discreteLemma} compute $D^*(S')$.
Set $\alpha_k=|Z_k|/\varphi(q(q-1))$, so $\sum\alpha_k=1$
and this, together with Lemma \ref{numberLemma}, implies that we may apply Lemma \ref{discreteLemma} to conclude
\begin{equation}
    \nonumber
    n(q,N)D^*(S')\le\max\left\{\frac{|Z_0|n(q,N)}{\varphi(q(q-1))},\frac{n(q,N)}{q}\right\}+\frac{CN}{q^2\log(N)}
\end{equation}
Further, we have
\begin{equation}
    \nonumber
    \frac{|Z_0|}{\varphi(q(q-1))}=\frac{(q-1)m_b(q)}{(q-1)\varphi(q-1)}=\frac{|\mathcal{N}|}{\varphi(q-1)}.
\end{equation}
Clearly,
$$
|\mathcal{N}|\le\frac{q-1}{{\rm ord}_q(b)}
$$
and it is well-known that
$$
\varphi(q-1)\ge\frac{C(q-1)}{\log\log(q-1)}
$$
Taking into consideration ${\rm ord}_q(b)\ge C\log(q)$, we get
\begin{eqnarray}
    \nonumber
    n(q,N)D^*(S')&\le&\max\left\{\frac{Cn(q,N)\log\log(q)}{\log(q)},\frac{n(q,N)}{q}\right\}+\frac{CN}{q^2\log(N)}\\
    \label{mainLemmaEq2}
    &=&\frac{Cn(q,N)\log\log(q)}{\log(q)}+\frac{CN}{q^2\log(N)}
\end{eqnarray}
By (\ref{mainLemmaEq1}) and (\ref{mainLemmaEq2}), we get
\begin{equation}
    \nonumber
    n(q,N)D^*(\mathcal{S}_b(\mathcal{F}_{q,N})\le \frac{Cn(q,N)\log\log(q)}{\log(q)}+\frac{CN}{q^2\log(N)},
\end{equation}
concluding the proof of (\ref{mainDiscrEstimate}).

\end{proof}

\begin{proof}[Proof of Theorem \ref{discrepancyThm}]
    Fix $N>N_0$ and set $M=|\mathcal{A}_N|$, it was shown in \cite{CKLRS} that
    \begin{equation}
        \label{Msize}
        M\sim\frac{N\log\log\log(N)}{\log(N)}.
    \end{equation}
    (We write $A\sim B$ if $c_1A\le B\le c_2A$ for absolute constants $c_1,c_2$.)
    Using Lemma \ref{triangleIneq} and Lemma \ref{discrepancyLemma}, we obtain
    \begin{eqnarray}
        \nonumber
        MD^*(\mathcal{S}_b(\mathcal{A}_N))&\le&\sum_{q\in\mathbb{P}}n(q,N)D^*(\mathcal{S}_b(\mathcal{F}_{q,N}))\\
        \label{discrEstProof1}
        &\le&C\sum_{q\in\mathbb{P}}\left(\frac{\log\log(q)}{\log(q)}n(q,N)+\frac{N}{q^2\log(N)}\right),
    \end{eqnarray}
    where $n(q,N)=0$ if $F_{q,N}=\emptyset$.
    By the prime number theorem
    \begin{equation}
        \label{discrEstProof2}
        n(q,N)\le\pi(N/q)=\frac{N}{q\log(N/q)}+\mathcal{O}\left(\frac{N}{q(\log(N/q))^2}\right)
    \end{equation}
    Using (\ref{discrEstProof1}), (\ref{discrEstProof2}) and the fact that $\log(N)/2\le \log(N/q)\le\log(N)$ for every $q\in\mathbb{P}$ with $n(q,N)>0$, we get
    \begin{eqnarray}
        \label{discrEstProof3}
        MD^*(\mathcal{S}_b(\mathcal{A}_N))&\le& 
        \frac{CN}{\log(N)}\sum_{q\in\mathbb{P}}\left(\frac{\log\log(q)}{q\log(q)}\left(1+\mathcal{O}\left(\frac{1}{\log(N)}\right)\right)+\frac{1}{q^2}\right).
    \end{eqnarray}
    Since the series $\sum_{q\in\mathbb{P}}\log\log(q)/(q\log(q))$ and $\sum_{q\in\mathbb{P}} 1/q^2$ both are convergent, it follows from (\ref{discrEstProof3}) that
    \begin{equation}
        \label{discrEstProof4}
        MD^*(\mathcal{S}_b(\mathcal{A}_N))\le\frac{CN}{\log(N)}\sum_{q\in\mathbb{P}}\left(\frac{\log\log(q)}{q\log(q)}+\frac{1}{q^2}\right)\le\frac{CN}{\log(N)}.
    \end{equation}
    From (\ref{discrEstProof4}) and (\ref{Msize}), we have
    $$
        D^*(\mathcal{S}_b(\mathcal{A}_N))=\mathcal{O}\left(\frac{1}{\log\log\log(N)}\right).
    $$
\end{proof}

{\bf Acknowledgements} The author is grateful to Professor A. Dubickas (Vilnius) for pointing out the references \cite{CKLRS, D2}.

\end{document}